\documentclass{amsart}
\usepackage{amsmath,amscd,amssymb}
\usepackage{amsfonts}
\usepackage[all]{xy}
\usepackage{enumerate}
\usepackage{color}
\numberwithin{equation}{subsection}
\theoremstyle{plain}
\newtheorem{thm}[subsection]{Theorem}

\newtheorem{lemma}[subsection]{Lemma}
\newtheorem{cor}[subsection]{Corollary}

\theoremstyle{definition}
\newtheorem{defn}[subsection]{Definition}
\newtheorem{notn}[subsection]{Notation}

\newtheorem{ackn}[subsection]{Acknowledgement}
\theoremstyle{remark}
\newtheorem{rem}[subsection]{Remark}
\newtheorem{conclrem}[subsection]{Concluding remarks}

\newtheorem{example}[subsection]{Example}

\begin{document}
\title{Iterated line integrals over Laurent series fields of characteristic $p$}
\author{Ambrus P\'al}
\date{March 16, 2017.}
\address{Department of Mathematics, 180 Queen's Gate, Imperial College, London, SW7 2AZ, United Kingdom}
\email{a.pal@imperial.ac.uk}
\begin{abstract} Inspired by Besser's work on Coleman integration, we use $\nabla$-modules to define iterated line integrals over Laurent series fields of characteristic $p$ taking values in double cosets of unipotent  $n\times n$ matrices with coefficients in the Robba ring divided out by unipotent $n\times n$ matrices with coefficients in the bounded Robba ring on the left and by unipotent $n\times n$ matrices with coefficients in the constant field on the right. We reach our definition by looking at the analogous theory for Laurent series fields of characteristic $0$ first, and reinterpreting the classical formal logarithm in terms of $\nabla$-modules on formal schemes. To illustrate that the new $p$-adic theory is non-trivial, we show that it includes the $p$-adic formal logarithm as a special case.  
\end{abstract}
\maketitle
\pagestyle{myheadings}
\markboth{Ambrus P\'al}{Line integrals over Laurent series fields}

\section{Formal iterated line integrals over Laurent series fields of characteristic zero}

In order to motivate our investigations over fields of positive characteristic, first we will look at a theory which could be justifiably considered as a formal analogue of line integrals over Laurent series fields of characteristic zero. We will start with the formal analogue of the logarithm, the most basic such contruction. Let $k$ a field of characteristic $0$.  The formal logarithm:
$$\log(1-z)=-\sum_{n=1}^{\infty}\frac{z^n}{n}\in\mathbb Q[[z]]$$
can be used to define a homomorphism:
$$k[[t]]^*/k^*\longrightarrow k[[t]]$$
as follows. Every $u\in k[[t]]^*$ can be written uniquely as:
$$u=c(1-w),\quad c\in k^*,\ w\in tk[[t]].$$
The infinite sum:
$$\log(1-w)=-\sum_{n=1}^{\infty}\frac{w^n}{n}$$
converges in the $t$-adic topology to a power series in $k[[t]]$, and the map:
$$k[[t]]^*\to k[[t]],\quad u\mapsto\log(1-w)$$
is a homomorphism with kernel $k^*$ which we will denote by log by slight abuse of notation.

It is possible to reinterpret this construction using differential algebra. Let $\Omega^1_{k[[t]]/k}$ be module of continuous K\"ahler differentials of $k[[t]]$ over $k$, i.e.~ the free module over
$k[[t]]$ generated by the symbol $dt$, where the derivation $d:k[[t]]\to\Omega^1_{k[[t]]/k}$ is given by the formula
$$d\big( \sum_{j=0}^{\infty}x_j t^j \big) =
\big( \sum_{j=1}^{\infty}jx_j t^{j-1} \big)\,dt.$$
Then the first de Rham cohomology group
$$H^1_{dR}(k[[t]])\stackrel{\textrm{def}}{=}
\Omega^1_{k[[t]]/k}/dk[[t]]$$
of $k[[t]]$ is trivial. Therefore for every $u\in k[[t]]^*$ there is a unique $v\in tk[[t]]$ such that
$$dv=\frac{du}{u}.$$
Note that $v=\log(u)$. Indeed this follows at once by differentiating the infinite sum term by term and using that $d$ is continuous in the $t$-adic topology. So the relation:
$$d\log(u)=\frac{du}{u}$$
can be used to define the formal logarithm. Next we give a geometric reformulation of this relation using the theory of
$\nabla$-modules.
\begin{defn}\label{nabla1} A $\nabla$-module over $k[[t]]$ is a pair $(M,\nabla)$, where $M$ is a finite, free $k[[t]]$-module, and $\nabla$ is a connection on $M$, i.e.~a $k$-linear map:
$$\nabla:M\to M\otimes_{k[[t]]}\Omega^1_{k[[t]]/k}$$
satisfying the Leibniz rule
$$\nabla(c\mathbf v)=c\nabla(\mathbf v)+\mathbf v\otimes dc\ \ 
(\forall c\in k[[t]],\mathbf v\in M).$$ 
The trivial $\nabla$-module over $k[[t]]$ is just the pair $(k[[t]],d)$. A horizontal map from a $\nabla$-module $(M,\nabla)$ to another $\nabla$-module $(M',\nabla')$ is just a $k[[t]]$-linear map
$f:M\to M'$ such that the following diagram is commutative:
$$\xymatrix{
M \ar^-{\nabla}[r] \ar^{f}[d] & M\otimes_{k[[t]]}\Omega^1_{k[[t]]/k} \ar^{f\otimes_{k[[t]]}\textrm{\rm id}_{\Omega^1_{k[[t]]/k}}}[d] \\
M'\ar^-{\nabla'}[r] & M'\otimes_{k[[t]]}\Omega^1_{k[[t]]/k}.}$$
As usual we will simply denote by $M$ the ordered pair $(M,\nabla)$ whenever this is convenient.
\end{defn}
These objects form a $k$-linear Tannakian category, with respect to horizontal maps as morphisms, and with the obvious notion of directs sums, tensor products, quotients and duals. In fact this Tannakian category is neutral, and the fibre functor is supplied by the lemma below. 
\begin{defn} A horizontal section of a $\nabla$-module
$(M,\nabla)$ over $k[[t]]$ is an $s\in M$ such that $\nabla(s)=0$. We denote the set of the latter by $M^{\nabla}$.
\end{defn}
The following claim is very well-known:
\begin{lemma}\label{solvability} For every $(M,\nabla)$ as above $M^{\nabla}$ is a $k$-linear vector space of dimension equal to the rank of $M$ over $k[[t]]$.
\end{lemma}
\begin{proof} See the proof of Theorem 7.2.1 of \cite{Ke} on page 121. Note that the recurrence
$$(i+1)U_{i+1}=\sum_{j=0}^iN_jU_{i-j}$$
has a solution in our case, too, since $k$ has characteristic zero. 
\end{proof}
Note that for every $s\in M^{\nabla}$ there is a unique morphism from the trivial $\nabla$-module to $(M,\nabla)$ such that the image of $1$ is $s$. Therefore the lemma above implies that every $\nabla$-module over $k[[t]]$ is {\it trivial}, i.e.~it is isomorphic to the $n$-fold direct sum of the trivial $\nabla$-module for some $n$.
In fact we get more:
\begin{cor}\label{equivalence} The functor
$$(M,\nabla)\mapsto M^{\nabla}$$
is a $k$-linear tensor equivalence of between the Tannakian categories of $\nabla$-modules over $k[[t]]$ and of finite dimensional $k$-linear vector spaces.
\end{cor}
\begin{proof} Since it is hard to find a convenient reference, we indicate the proof  for the sake of the reader. Let $F$ be the functor in the claim above, and let $G$ denote the functor 
$$V\mapsto (V\otimes_kk[[t]],\textrm{id}_V\otimes_kd)$$
from the category of finite dimensional $k$-linear vector spaces to the category of $\nabla$-modules over $k[[t]]$. It is easy to see that $F$ and $G$ are functors of $k$-linear tensor categories, so we only need to see that they are equivalences of categories.  Note that the $k[[t]]$-multiplication induces a natural map
$$M^{\nabla}\otimes_kk[[t]]\longrightarrow M$$
which is an isomorphism by Lemma \ref{solvability}. Similarly the natural map
$$V\longrightarrow
(V\otimes_kk[[t]])^{\textrm{id}_V\otimes_kd}$$
given by the rule $v\mapsto v\otimes_k1$ is an isomorphism.
\end{proof}
We will need a slight variant of Lemma \ref{solvability}, taking into accounts filtrations, but this will follow easily from Corollary \ref{equivalence}.
\begin{notn}\label{1.5} Let $M$ be a $\nabla$-module over $k[[t]]$ equipped with a filtration:
$$0=M_0\subset M_1\subset\cdots\subset M_n=M$$
by sub $\nabla$-modules such that the rank of $M_i$ over $k[[t]]$ is $r_1+\cdots+r_i$. Set $r=r_1+r_2+\cdots+r_n$, and equip the trivial $\nabla$-module $T=k[[t]]^{\oplus r}$ with the filtration:
$$0=T_0\subset T_1\subset\cdots\subset T_n=T,$$
where 
$$T_i=\underbrace{k[[t]]\oplus k[[t]]\oplus\cdots\oplus k[[t]]}
_{r_1+\cdots+r_i}\oplus\underbrace{0\oplus\cdots\oplus0}
_{r_{i+1}+\cdots+r_n}.$$
\end{notn}
\begin{lemma} There is an isomorphism $\phi:M\to T$ of $\nabla$-modules such that $\phi(M_i)=T_i$ for every index $i=1,2,\ldots,n$. 
\end{lemma}
\begin{proof} By taking horizontal sections we get a filtration:
$$0=M_0^{\nabla}\subset M_1^{\nabla}\subset\cdots\subset M_n^{\nabla}=M^{\nabla}$$
of $M^{\nabla}$ by $k$-linear subspaces such that the $k$-dimension of $M_i^{\nabla}$ is $r_1+\cdots+r_i$ by Lemma \ref{solvability}. Similarly
$$0=T_0^{\nabla}\subset T_1^{\nabla}\subset\cdots
\subset T_n^{\nabla}=T^{\nabla}$$
is a filtration of $T^{\nabla}$ such that the $k$-dimension of $T_i^{\nabla}$ is $r_1+\cdots+r_i$. It is a basic fact of linear algebra that there is a $k$-linear isomorphism $f:M^{\nabla}\to
T^{\nabla}$ such that $f(M_i^{\nabla})=T_i^{\nabla}$. The claim now follows from Corollary \ref{equivalence}.
\end{proof}
Let $M$ and $T$ be as in Notation \ref{1.5}. Assume now that for every index $i=1,2,\ldots,n$ an isomorphism:
$$\phi_i:M_i/M_{i-1}\longrightarrow k[[t]]^{\oplus r_i}$$
is given where $k[[t]]$ is equipped with the trivial connection. 
\begin{lemma}\label{filtered_iso} There is an isomorphism $\phi:M\to T$ of $\nabla$-modules such that $\phi(M_i)=T_i$ for every index $i=1,2,\ldots,n$ and the induced isomorphism
$$\phi^i:M_i/M_{i-1}\longrightarrow T_i/T_{i-1}\cong k[[t]]^{\oplus r_i}$$
is $\phi_i$ for every index $i=1,2,\ldots,n$. 
\end{lemma}
\begin{proof} Let
$$\phi_i^{\nabla}:(M_i/M_{i-1})^{\nabla}\cong
M_i^{\nabla}/M_{i-1}^{\nabla}
\longrightarrow T^{\nabla}_i/T^{\nabla}_{i-1}\cong
(T_i/T_{i-1})^{\nabla}\cong k^{\oplus r_i}$$
be the $k$-linear isomorphism induced by $\phi_i$ on horizontal sections. It is possible to choose a a $k$-linear isomorphism $f:M^{\nabla}\to T^{\nabla}$ such that
$f(M_i^{\nabla})=T_i^{\nabla}$ and the induced map:
$$M_i^{\nabla}/M_{i-1}^{\nabla}
\longrightarrow T^{\nabla}_i/T^{\nabla}_{i-1}$$
is $\phi^{\nabla}_i$ above for every index $i=1,2,\ldots,n$. The claim now follows from Corollary \ref{equivalence}.
\end{proof}
\begin{defn}\label{framed1} Let $\underline r=(r_1,r_2,\ldots,r_n)$ be a vector consisting of positive integers, and set $r=r_1+r_2+\cdots+r_n$. A framed $\nabla$-module of signature $\underline r$ is a $\nabla$-module $(M,\nabla)$ over $k[[t]]$ equipped with a $k[[t]]$-basis $e_1,e_2,\ldots,e_r$ of $M$ such that
$$M_i=\textrm{ the $k[[t]]$-span of $e_1,e_2,\ldots,e_{r_1+\cdots+r_i}$}$$
is a sub $\nabla$-module, and the image of $e_{r_1+\cdots+r_{i-1}+1},\ldots,e_{r_1+\cdots+r_i}$ in the quotient $M_i/M_{i-1}$ is a $k$-basis of $(M_i/M_{i-1})^{\nabla}$. There is a natural notion of isomorphism of framed $\nabla$-modules of signature $\underline r$, namely, it is an isomorphism of the underlying $\nabla$-modules which maps the $k[[t]]$-bases to each other (respecting the indexing, too). 
\end{defn}
\begin{defn} Let $R$ be a commutative ring with unity. Let $U_{\underline r}(R)$ denote the group of $r\times r$ matrices composed of blocks $U_{ij}$ such that for every pair $(i,j)$ of indices $U_{ij}$ is an $r_i\times r_j$ matrix with coefficients in $R$, moreover $U_{ii}$ is the identity matrix for every $i$ and $U_{ij}$ is the zero matrix for every $i>j$. It is reasonable to call $U_{\underline r}(R)$ the group of unipotent matrices of rank $\underline r$ with coefficients in $R$. 
\end{defn}
\begin{rem} Note that for every framed $\nabla$-module
$(M,\nabla,e_1,e_2,\ldots,e_r)$ of signature $\underline r$ as above there is a unique isomorphism:
$$\phi_i:M_i/M_{i-1}\longrightarrow k[[t]]^{\oplus r_i}$$
which maps the the image of $e_{r_1+\cdots+r_{i-1}+1},\ldots,e_{r_1+\cdots+r_i}$ under the quotient map to the 1st, 2nd,\ldots,$r_i$th basis vector of $k[[t]]^{\oplus r_i}$, respectively. Therefore there is an isomorphism $\phi:M\to T$ of $\nabla$-modules such that $\phi(M_i)=T_i$ and the induced isomorphism
$$\phi^i:M_i/M_{i-1}\longrightarrow T_i/T_{i-1}\cong k[[t]]^{\oplus r_i}$$
is $\phi_i$ for every index $i=1,2,\ldots,n$ by Lemma \ref{filtered_iso}. The matrix of $\phi$ in the basis $e_1,e_2,\ldots,e_r$ is an element of
$U_{\underline r}(k[[t]])$, unique up to multiplication on the right by a matrix in $U_{\underline r}(k)$. We get a well-defined map from the isomorphism classes of framed  $\nabla$-modules of signature $\underline r$ into the set $U_{\underline r}(k[[t]])/U_{\underline r}(k)$ which is obviously a bijection. 
\end{rem}
\begin{example}\label{nabla-log} For every $u\in k[[t]]^*$ consider the following framed $\nabla$-module of signature $(1,1)$. Set $M=k[[t]]^{\oplus 2}$, let $e_1,e_2$ be the 1st, respectively 2nd basis vector of $M$, and let $\nabla$ be the unique connection of $M$ such that
$$\nabla(e_1)=0,\quad\nabla(e_2)=e_1\otimes\frac{du}{u}.$$
Let $\phi:M\cong k[[t]]^{\oplus 2}\to T\cong k[[t]]^{\oplus 2}$ be an isomorphism of the type considered above. Then the matrix $V$ of $\phi$ in the basis $e_1,e_2$ is
$$V=\begin{pmatrix} 1& v \\ 0 & 1\end{pmatrix}\in U_{(1,1)}(k[[t]])
\textrm{ such that}$$
$$d\circ V=\begin{pmatrix} 0& dv \\ 0 & 0\end{pmatrix}=V\circ\nabla=\begin{pmatrix} 1& v \\ 0 & 1\end{pmatrix}\cdot\begin{pmatrix} 0&\frac{du}{u}  \\ 0 & 0\end{pmatrix}=
\begin{pmatrix} 0& \frac{du}{u} \\ 0 & 0\end{pmatrix},$$
and hence
$$dv=\frac{du}{u}.$$
So the isomorphism class of the framed $\nabla$-module $(M,\nabla,e_1,e_2)$ in
$$U_{(1,1)}(k[[t]])/U_{(1,1)}(k)\cong k[[t]]/k$$
is just $\log(u)$ (modulo constants).
\end{example}
The point of the construction above is that we can get the family in the example above as a pull-back of a similar type of object on the formal multiplicative group scheme over the formal spectrum Spf$(k[[t]])$ of $k[[t]]$. This is the description which easily generalises, and which we are going to describe next.
\begin{defn} Let $X$ be a formally smooth $t$-adic formal scheme of finite type over Spf$(k[[t]])$. Then $X$ is also a formally smooth  formal scheme of finite type over Spf$(k)$ via the map Spf$(k)
\to\textrm{Spf}(k[[t]])$ induced by the embedding $k\hookrightarrow k[[t]]$. Therefore the sheaf of continuous K\"ahler differentials
$\Omega^1_{X/k}$ is well-defined, and it is a finite, locally free formal $\mathcal O_X$-module. A $\nabla$-module over $X$ is a pair $(M,\nabla)$, where $M$ is a finite, locally free formal
$\mathcal O_X$-module, and $\nabla$ is a connection on $M$, i.e.~a $k$-linear map of sheaves:
$$\nabla:M\to M\otimes_{\mathcal O_X}\Omega^1_{X/k}$$
satisfying the Leibniz rule
$$\nabla(c\mathbf v)=c\nabla(\mathbf v)+\mathbf v\otimes dc$$
for every open $U\subset X$ and $c\in\Gamma(U,\mathcal O_X),\mathbf v\in\Gamma(U,M)$. 
\end{defn}
\begin{defn} The trivial $\nabla$-module over $X$ is just $\mathcal O_X$ equipped with the differential $d:\mathcal O_X\to\Omega^1_{X/k}\cong\mathcal O_X\otimes_{\mathcal O_X}\Omega^1_{X/k}$. These notions specialise to those introduced in Definition \ref{nabla1} when $X$ is Spf$(k[[t]])$. Moreover horizontal maps of $\nabla$-modules over $X$ is defined the same way as above. We get a $k$-linear category with the usual notion of direct sums, duals and tensor products. Again we will denote by $M$ the ordered pair $(M,\nabla)$ whenever this is convenient. Finally let $M^{\nabla}$ denote the sheaf of horizontal sections of $M$:
$$\Gamma(U,M^{\nabla})\stackrel{\textrm{def}}{=}
\{s\in\Gamma(U,M)\mid\nabla(s)=0\}.$$
Note that $M$ is a trivial $\nabla$-module of rank $n$, that is, isomorphic to the $n$-fold direct sum of $(\mathcal O_X,d)$, if and only if $M^{\nabla}$ is the constant sheaf in $n$-dimensional $k$-linear vector spaces. 
\end{defn}
\begin{defn} It is possible to define the notion of framed $\nabla$-modules in this more general context, too. Let $\underline r$ and
$r$ be as in Definition \ref{framed1}. A framed $\nabla$-module over $X$ of signature $\underline r$ is a $\nabla$-module $(M,\nabla)$ over $X$ equipped with a $\mathcal O_X$-frame $e_1,e_2,\ldots,e_r$ of $M$ such that
$$M_i=\textrm{ the $\mathcal O_X$-span of $e_1,e_2,\ldots,e_{r_1+\cdots+r_i}$}$$
is a sub $\nabla$-module, and the image of $e_{r_1+\cdots+r_{i-1}+1},\ldots,e_{r_1+\cdots+r_i}$ in the quotient $M_i/M_{i-1}$ is a $k$-frame of $(M_i/M_{i-1})^{\nabla}$. 
\end{defn}
\begin{defn} The notion of $\nabla$-modules and framed $\nabla$-modules are natural in $X$. Let $f:X\to Y$ be a morphism of formally smooth formal schemes of finite type over Spf$(k[[t]])$. The morphism $f$ induces an $\mathcal O_X$-linear map $df:f^*(\Omega^1_{Y/k})\to\Omega^1_{X/k}$. The pull-back
$f^*(M,\nabla)$ of a $\nabla$-module $(M,\nabla)$ with respect to $f$ is $f^*(M)$ equipped with the composition:
$$\xymatrix{
f^*(\nabla):f^*(M)\ar[r] & f^*(M\otimes_{\mathcal O_Y}\Omega^1_{Y/k})\cong f^*(M)\otimes_{\mathcal O_X}f^*(\Omega^1_{Y/k})
\ar[r]
& \Omega^1_{X/k},}$$
where the first arrow is the pull-back of $\nabla$ with respect to $f$, and the second is $\textrm{id}_{f^*(M)}\otimes_{\mathcal O_X}df$. The pull-back of a framed $\nabla$-module $(M,\nabla,e_1,\ldots,e_r)$ of signature $\underline r$ on $Y$ with respect to $f$ is the pull-back $f^*(M,\nabla)$ equipped with the $\mathcal O_X$-frame $f^*(e_1),\ldots,f^*(e_r)$. Since pull-back commutes with quotients and the pull-back of horizontal sections are horizontal, this construction is a framed $\nabla$-module of signature $\underline r$ on $X$.
\end{defn}
\begin{defn} For every $X$ as above let $X(k[[t]])$ denote the set of sections $f:\textrm{Spf}(k[[t]])\to X$. Let $\mathbf M=(M,\nabla,e_1,\ldots,e_r)$ be a framed $\nabla$-module of signature $\underline r$ on $X$. Then for every $f\in X(k[[t]])$ the pull-back of $\mathbf M$ with respect to $f$ is a framed $\nabla$-module of signature $\underline r$ over $k[[t]]$. Taking isomorphism classes we get a function
$$\int_{\mathbf M}:X(k[[t]])\longrightarrow U_{\underline r}(k[[t]])
/U_{\underline r}(k)$$
which we will call the line integral of $\mathbf M$. 
\end{defn}
\begin{example}\label{nabla_global} Let $X$ be Spf$(k[[t,x]])$. In order to give a $\nabla$-module on $X$, it is sufficient to give a $k$-linear map:
$$\nabla:k[[t,x]]^{\oplus2}\longrightarrow 
k[[t,x]]^{\oplus2}\otimes_{k[[t,x]]}\Omega^1_{k[[t,x]]/k}$$
satisfying the Leibniz rule, where
$$\Omega^1_{k[[t,x]]/k}=k[[t,x]]\cdot dt\oplus k[[t,x]]\cdot dx,$$
with differential $d:k[[t,x]]\to\Omega^1_{k[[t,x]]/k}$ given by:
$$d\big(\sum_{ij}a_{ij}t^ix^j\big)=\sum_{ij}
(ia_{ij}t^{i-1}x^jdt+ja_{ij}t^ix^{j-1}dx).$$
Let $e_1,e_2$ be the 1st, respectively 2nd basis vector of $k[[t,x]]^{\oplus2}$, and let $\nabla$ be the unique connection of $k[[t,x]]^{\oplus2}$ such that
$$\nabla(e_1)=0,\quad\nabla(e_2)=e_1\otimes\frac{dx}{1+x},$$
where $(1+x)^{-1}=\sum_{i=0}^{\infty}(-1)^ix^i$. Equipped with the frame $e_1,e_2$ this $\nabla$-module is framed of signature $(1,1)$. Let $\mathbf M$ denote this object. Note that sections of $X\to\textrm{Spf}(k[[t]])$ are exactly continuous $k[[t]]$-algebra homomorphisms $\psi:k[[t,x]]\to k[[t]]$. Every such $\psi$ is determined by $\psi(1+x)$ which must be an invertible element of $k[[t]]$. Conversely for every $u\in k[[t]]^*$ there is a unique such $\psi_u:k[[t,x]]\to k[[t]]$ with the property $\psi_u(1+x)=u$. The  pull-back of $\mathbf M$ with respect to $\psi_u$ is just the framed $\nabla$-module appearing in Example \ref{nabla-log}. We get that the formal line integral:
$$\int_{\mathbf M}:X(k[[t]])\cong k[[t]]^*\longrightarrow U_{(1,1)}(k[[t]])
/U_{(1,1)}(k)\cong k[[t]]/k$$
is just the formal logarithm.
\end{example}

\section{The $p$-adic logarithm for Laurent series fields of characteristic $p$}

The perfect reference for the background material in this section and the next is Kedlaya's book \cite{Ke}. 
\begin{notn} Let $k$ a perfect field of characteristic $p>0$ and let $\mathcal O$ denote the ring of Witt vectors over $k$.  Let $v_p$ denote the valuation on $\mathcal O$ normalised so that $v_p(p) = 1$. For $x \in \mathcal O$, let $\overline{x}$ denote its reduction in $k$. Let $\Gamma$ denote the ring of bidirectional power series:
$$\Gamma=\big\{\sum_{i \in \mathbb Z} x_i u^i\mid x_i \in \mathcal O,\ \ 
\lim_{i \to -\infty}v_p(x_i)=\infty\big\}.$$
Then $\Gamma$ is a complete discrete valuation ring whose residue field we could identify with $k((t))$ by identifying the reduction of $\sum x_i u^i$ with $\sum \overline{x}_it^i$ (see page 263 of \cite{Ke}). Let $K=\mathcal O[\frac{1}{p}]$ and $\mathcal E=\Gamma[\frac{1}{p}]$; they are the fraction fields of the rings $\mathcal O$ and $\Gamma$, respectively.
\end{notn}
\begin{defn} Let $\Omega^1_{\mathcal E}$ be the free module over $\mathcal E$ generated by a symbol $du$, and define the derivation $d:\mathcal E\to \Omega^1_{\mathcal E}$ by the formula
$$d\big( \sum_j x_j u^j \big) = \big( \sum_j j x_j u^{j-1} \big)\,du.$$
We define the first de Rham cohomology group $H^1_{dR}(\mathcal E)$ of $\mathcal E$ as the quotient
$\Omega^1_{\mathcal E}/d\mathcal E$. 
Note that the dlog map: 
$$x\mapsto\frac{dx}{x},\quad\mathcal E^*\to \Omega^1_{\mathcal E}$$
followed by the quotient map $\Omega^1_{\mathcal E}\to H^1_{dR}(\mathcal E)$ furnishes a homomorphism
$\Gamma^*\to H^1_{dR}(\mathcal E)$ which we will denote by dlog by slight abuse of notation.
\end{defn}
\begin{lemma}\label{factors_thru1} The homomorphism $\textrm{\rm dlog}:\Gamma^*\to H^1_{dR}(\mathcal E)$ factors through the reduction map $\overline{\cdot}:\Gamma^*\to k((t))^*$.
\end{lemma}
\begin{proof} We need to show that for every $x\in\Gamma^*$ of the form $1-py$ with $y\in\Gamma$ we have dlog$(x)\in d\mathcal E$. Set
$$z=-\sum_{n=1}^{\infty}\frac{(py)^n}{n}.$$
Since $0\leq v_p(p^n)-v_p(n)\to\infty$ as $n\to\infty$, the infinite sum above converges in the $p$-adic topology, and hence $z\in\Gamma$ is well-defined. Differentiation is continuous with respect to the $p$-adic topology, so 
$$dz=\sum_{n=1}^{\infty}(py)^{n-1}d(-py)=(1-py)^{-1}d(1-py)=\textrm{dlog}(x).$$
\end{proof}
Let dlog also denote the induced homomorphism $k((t))^*\to H^1_{dR}(\mathcal E)$. This map is trivial restricted to $k^*$, for example because $\textrm{\rm dlog}:\Gamma^*\to H^1_{dR}(\mathcal E)$ is trivial on $\mathcal O^*$. The basic result about this construction is the following
\begin{thm}\label{p-adic_log1} The kernel of $\textrm{\rm dlog}:k((t))^*\to H^1_{dR}(\mathcal E)$ is $k^*$.
\end{thm}
\begin{proof} Let $\deg:k((t))^*\to\mathbb Z$ be the discrete valuation on $k((t))$ normalised so that $\deg(t)=1$. We define the residue map on $\Omega^1_{\mathcal E}$ as follows:
$$\sum_j x_j u^j du\mapsto x_{-1},\quad\Omega^1_{\mathcal E}\to K.$$
Since there is no term of degree $-1$ in any exact form $dx\in d\mathcal E$, we get a well-defined homomorphism res$:H^1_{dR}(\mathcal E)\to K$. We will need the following:
\begin{lemma}\label{residue} The diagram commutes:
$$\xymatrix{ k((t))^* \ar^-{\textrm{\rm dlog}}[r] \ar^{\deg}[d] & H^1_{dR}(\mathcal E)
\ar^{\textrm{\rm res}}[d] \\
\mathbb Z  \ar@{^{(}->}[r] & K.}$$
\end{lemma}
\begin{proof} Clearly $\textrm{res}\circ\textrm{dlog}(t)=1$. Now let $x\in k[[t]]^*$. Then $x$ has a lift to $(\Gamma_+)^*\subset\Gamma^*$, where $\Gamma_+$ denotes the subring
$$\Gamma_+=\big\{\sum_{i \in \mathbb N} x_i u^i\mid x_i \in \mathcal O\}$$
of $\Gamma$. By definition $\textrm{res}\circ\textrm{dlog}((\Gamma_+)^*)=0$. Since the group $k((t))^*$ is generated by $t$ and $k[[t]]^*$, the claim now follows, as all arrows in the diagram are homomorphisms.
\end{proof}
Let us return to the proof of Theorem \ref{p-adic_log1}. Let $x\in k((t))^*$ be such that dlog$(x)=0$, but $x\not\in k^*$. By the above $x\in k[[t]]^*$. We may assume without loss of generality that $x\in 1+tk[[t]]$ by multiplying $x$ with an element of $k^*$. Choose a lift $y\in(\Gamma_+)^*$ of $x$. We may assume that
$$y=1-au^m-bu^{m+1},$$
where $m$ is a positive integer, with $a\in\mathcal O^*$ and $b\in\Gamma_+$. Set
$$z=-\sum_{n=1}^{\infty}\frac{(au^m+bu^{m+1})^n}{n}.$$
The infinite sum above converges with respect to the topology generated by the ideal $(u)\triangleleft K[[u]]$, so $z$ is a well-defined element of $K[[u]]$.

Let $R$ be one of the rings $K[[t]]$ and $\mathcal E_+=\Gamma_+[\frac{1}{p}]$, and let $\Omega^1_R$ be the free module over $R$ generated by a symbol $du$, and define the derivation $d:R\to \Omega^1_R$ by the formula
$$d\big( \sum_j x_j u^j \big) = \big( \sum_j j x_j u^{j-1} \big)\,du.$$
Clearly $\Omega^1_{\mathcal E_+}\subset\Omega^1_{K[[t]]}$. Let $v\in\mathcal E$ be such that $dv=\textrm{dlog}(y)$. Since $\textrm{dlog}(y)\in\Omega^1_{\mathcal E_+}$ we have $v\in\mathcal E_+$. Note that differentiation is continuous with respect to the $(u)$-adic topology, so 
\begin{eqnarray}
dz &=& \sum_{n=1}^{\infty}(au^m+bu^{m+1})^{n-1}d(-au^m-bu^{m+1})
\nonumber\\
&=& (1-au^m-bu^{m+1})^{-1}d(1-au^m-bu^{m+1})=
\textrm{dlog}(y).\nonumber
\end{eqnarray}
Therefore $dv=dz$ and hence $v-z\in K$. We get that $z\in\mathcal E_+$, too. But this is a contradiction since, if 
$$z=\sum_{i=0}^{\infty}z_iu^i,$$ 
then $v_p(z_{mp^i})=-i$ for every positive integer $i$. We can see the latter as follows. By definition:
$$z\equiv-\sum_{n=1}^{p^i-1}\frac{(au^m+bu^{m+1})^n}{n}+
\frac{(au^m)^{p^i}}{p^i}\mod(u^{mp^i+1}).$$
In the first summand all coefficients have $p$-adic valuation
$\geq1-i$, while in the second the coefficient of $u^{mp^i}$ has valuation $-i$.
\end{proof}
Next we are going to give a slightly more convoluted variant of this construction, which nevertheless ties it up better with the general theory of line integrals over Laurent series fields of characteristic $p$.
\begin{defn} Let $\Gamma^{\dagger}$ denote the subring:
$$\Gamma^{\dagger}=\big\{\sum_{i \in \mathbb Z} x_i u^i\mid
x_i \in \mathcal O, \ \ \liminf_{i \to -\infty}\frac{v_p(x_i)}{-i} > 0\big\}
\subset\Gamma.$$
The latter is also a discrete valuation ring with residue field $k((t))$, although it is not complete (see Definition 15.1.2 and Lemma 15.1.3 of \cite{Ke} on page 263). Let $\mathcal E^{\dagger}=\Gamma^{\dagger}[\frac{1}{p}]$. Then $\mathcal E^{\dagger}$ is the fraction field of the ring $\Gamma^{\dagger}$. Similarly to the above let $\Omega^1_{\mathcal E^{\dagger}}$ be the module of continuous K\"ahler differentials of $\mathcal E^{\dagger}$, i.e.~the free module over $\mathcal E^{\dagger}$ generated by a symbol $du$, equipped with the derivation $d:\mathcal E^{\dagger}\to \Omega^1_{\mathcal E^{\dagger}}$ given by
$$d\big( \sum_j x_j u^j \big) = \big( \sum_j j x_j u^{j-1} \big)\,du.$$
We define the first de Rham cohomology group $H^1_{dR}(\mathcal E^{\dagger})$ of $\mathcal E^{\dagger}$ as the quotient
$\Omega^1_{\mathcal E^{\dagger}}/d\mathcal E^{\dagger}$. 
Note that the dlog map: 
$$x\mapsto\frac{dx}{x},\quad(\mathcal E^{\dagger})^*\to \Omega^1_{\mathcal E^{\dagger}}$$
followed by the quotient map $\Omega^1_{\mathcal E^{\dagger}}\to H^1_{dR}(\mathcal E^{\dagger})$ furnishes a homomorphism
$(\Gamma^{\dagger})^*\to H^1_{dR}(\mathcal E^{\dagger})$ which we will denote by $\textrm{dlog}^{\dagger}$.
\end{defn}
\begin{lemma} The homomorphism $\textrm{\rm dlog}^{\dagger}:(\Gamma^{\dagger})^*\to H^1_{dR}(\mathcal E^{\dagger})$ factors through the reduction map $\overline{\cdot}:(\Gamma^{\dagger})^*\to k((t))^*$.
\end{lemma}
\begin{proof} We need to show that for every
$x\in(\Gamma^{\dagger})^*$ of the form $1-py$ with $y\in\Gamma^{\dagger}$ we have dlog$(x)\in d\mathcal E^{\dagger}$. It will be sufficient to prove that the element
$$z=-\sum_{n=1}^{\infty}\frac{(py)^n}{n}\in\Gamma$$
is actually in $\Gamma^{\dagger}$. Note that $\mathcal E^{\dagger}$ is the ring of the bidirectional (or Laurent) expansions of bounded holomorphic functions over $K$ on an open annulus of outer radius $1$ and inner radius $1-\epsilon$, for some $\epsilon\in(0,1)$ (see page 263 of \cite{Ke}). If $y$ is such a function then the infinite sum defining
$z$ converges with respect to the supremum norm and defines a bounded holomorphic function over $K$ on the annulus of outer radius $1$ and inner radius $1-\epsilon$. The claim is now clear.
\end{proof}
Let $\textrm{dlog}^{\dagger}$ also denote the induced homomorphism $k((t))^*\to H^1_{dR}(\mathcal E^{\dagger})$. This map is trivial restricted to $k^*$, for example because $\textrm{\rm dlog}^{\dagger}:(\Gamma^{\dagger})^*\to H^1_{dR}(\mathcal E^{\dagger})$ is trivial on $\mathcal O^*$. Then we have the following variant of Theorem \ref{p-adic_log1} above:
\begin{thm}\label{p-adic_log2} The kernel of $\textrm{\rm dlog}^{\dagger}:k((t))^*\to H^1_{dR}(\mathcal E^{\dagger})$ is
$k^*$.
\end{thm}
\begin{proof} Note that there is a commutative diagram:
$$\xymatrix{ k((t))^* \ar^-{\textrm{\rm dlog}^{\dagger}}[r] \ar_{\textrm{\rm dlog}}[dr] & H^1_{dR}(\mathcal E^{\dagger})
\ar[d] \\
& H^1_{dR}(\mathcal E),}$$
where the right vertical map is induced by the pair of inclusions
$\Omega^1_{\mathcal E^{\dagger}}\to\Omega^1_{\mathcal E}$ and
$d\mathcal E^{\dagger}\to d\mathcal E$. Now the claim immediately follows from Theorem \ref{p-adic_log1}.
\end{proof}
\begin{defn} Let $\mathcal R$ denote the ring of bidirectional power series:
$$\mathcal R=\big\{\sum_{i \in \mathbb Z} x_i u^i\mid x_i \in 
\mathcal O[\frac{1}{p}], \ \ \liminf_{i \to -\infty}\frac{v_p(x_i)}{-i}>0,\ \ 
\liminf_{i \to+\infty}\frac{v_p(x_i)}{i}\geq 0\}.$$
(See Definition 15.1.4 of \cite{Ke} on page 264.) Let $\mathcal R_+$ denote its subring:
$$\mathcal R_+=\mathcal R\cap\big\{\sum_{i \in \mathbb N} x_i u^i\mid x_i \in 
\mathcal O[\frac{1}{p}]\big\}.$$
Clearly $\mathcal E_+\subset\mathcal R_+$ and $\mathcal E^{\dagger}
\subset\mathcal R$. Note that we may define the continuous
K\"ahler differentials and the first de Rham cohomology group of the rings $\mathcal R$ and $\mathcal R_+$ similarly to the above, and we will use similar notation to denote them, too.
\end{defn}
The reason we like the ring $\mathcal R_+$ is the following very well-known claim:
\begin{lemma}\label{solvability2} The group $H^1_{dR}(\mathcal R_+)$ is trivial.
\end{lemma}
\begin{proof} Simply note that if
$\sum_{i=0}^{\infty}x_i u^i\in\mathcal R_+$ then $\sum_{i=0} ^{\infty}\frac{x_i}{i+1}u^{i+1}$ also lies in $\mathcal R_+$.
\end{proof}
Now we can tie in the contents of this section with the formal logarithm construction of the previous section.
\begin{defn} Let $v\in k[[t]]^*$. Then $\textrm{dlog}^{\dagger}(v)\in H^1_{dR}(\mathcal E_+)$. By the above the image of this class under the natural map $H^1_{dR}(\mathcal E_+)\to H^1_{dR}(\mathcal R_+)$ is trivial, so there is a $w\in\mathcal R_+$ such that $dw=\textrm{dlog}^{\dagger}(v)$, unique up to adding an element of $\mathcal E_+$. It is reasonable to denote the class of this element in $\mathcal R_+/\mathcal E_+$ by $\textrm{log}^{\dagger}(v)$ in light of the above. The resulting map $\textrm{log}^{\dagger}:k[[t]]^*\to\mathcal R_+/\mathcal E_+$ is a homomorphism with kernel $k^*$. 
\end{defn}
\begin{rem} There is an obstruction to extend this construction to the whole $k((t))^*$, taking values in $\mathcal R/\mathcal E^{\dagger}$, namely the residue map. Indeed similarly to the construction in the proof of Theorem \ref{p-adic_log1}, there is a  residue map on $\Omega^1_{\mathcal E^{\dagger}}$ given by
$$\sum_j x_j u^j du\mapsto x_{-1},\quad\Omega^1_{\mathcal E^{\dagger}}\to K,$$
moreover we have a similar map for $\Omega^1_{\mathcal R}$, and these maps are compatible with the inclusions $\Omega^1_{\mathcal E^{\dagger}}\subset\Omega^1_{\mathcal E}$ and $\Omega^1_{\mathcal E^{\dagger}}\subset\Omega^1_{\mathcal R}$. Since there is no term of degree $-1$ in any exact form, we get well-defined homomorphisms res$:H^1_{dR}(\mathcal E^{\dagger})\to K$ and res$:H^1_{dR}(\mathcal R)\to K$. From Lemma \ref{residue} we get that the diagram commutes:
$$\xymatrix{ k((t))^* \ar^-{\textrm{\rm dlog}^{\dagger}}[r] \ar^{\deg}[d] & H^1_{dR}(\mathcal E^{\dagger})
\ar^{\textrm{\rm res}}[d]\ar[r] & H^1_{dR}(\mathcal R)
\ar^{\textrm{\rm res}}[dl] \\
\mathbb Z  \ar@{^{(}->}[r] & K .&}$$
On the other hand the map
$$\textrm{res}:H^1_{dR}(\mathcal R)\longrightarrow K$$
is an isomorphism by the lemma below, so $\textrm{\rm dlog}^{\dagger}(v)$ is integrable if and only if $v\in k[[t]]^*$.
\end{rem}
\begin{lemma}\label{solvability3} The map
$\textrm{\rm res}:H^1_{dR}(\mathcal R)\longrightarrow K$
is an isomorphism.
\end{lemma}
\begin{proof} The map is obviously surjective. In order to see injectivity, simply note that if $\sum_{i \in \mathbb N,i\neq-1}x_i u^i\in\mathcal R$ then $\sum_{i \in \mathbb N,i\neq-1}\frac{x_i}{i+1}u^{i+1}$ also lies in $\mathcal R$.
\end{proof}

\section{Iterated $p$-adic line integrals over Laurent series fields of characteristic $p$}

\begin{defn}\label{nabla1b} Let $R$ be one of the rings $\mathcal E_+,\mathcal E,\mathcal E^{\dagger},\mathcal R_+$ or $\mathcal R$. A $\nabla$-module over $R$ is a pair $(M,\nabla)$, where $M$ is a finite, free $R$-module, and $\nabla$ is a connection on $M$, i.e.~a $K$-linear map:
$$\nabla:M\to M\otimes_{R}\Omega^1_{R}$$
satisfying the Leibniz rule
$$\nabla(c\mathbf v)=c\nabla(\mathbf v)+\mathbf v\otimes dc\ \ 
(\forall c\in R,\mathbf v\in M).$$ 
The trivial $\nabla$-module over $R$ is just the pair $(R,d)$. A horizontal map from a $\nabla$-module $(M,\nabla)$ to another $\nabla$-module $(M',\nabla')$ is just a $R$-linear map
$f:M\to M'$ such that the following diagram is commutative:
$$\xymatrix{
M \ar^-{\nabla}[r] \ar^{f}[d] & M\otimes_{R}\Omega^1_{R} \ar^{f\otimes_{R}\textrm{\rm id}_{\Omega^1_{R}}}[d] \\
M'\ar^-{\nabla'}[r] & M'\otimes_{R}\Omega^1_{R}.}$$
As usual we will simply denote by $M$ the ordered pair $(M,\nabla)$ whenever this is convenient.
\end{defn}
\begin{defn} Now let $R\subset R'$ be two rings from the list above and let $(M,\nabla)$ be a $\nabla$-module over $R$. Let $\nabla'$ be the unique connection:
$$\nabla':M\otimes_RR'\longrightarrow
 (M\otimes_RR')\otimes_{R'}\Omega^1_{R'}\cong (M\otimes_R\Omega^1_{R})\otimes_{R'}R'$$
such that
$$\nabla'(m\otimes_Rs)=\nabla m\otimes_Rs+m\otimes_Rds,\ \ 
(\forall m\in M,\forall s\in R).$$ 
Then the couple $(M\otimes_RR',\nabla')$ is a $\nabla$-module over $R'$ which we will denote by $M\otimes_RR'$ for simplicity and will call the pull-back of $M$ onto $R'$. Moreover for every homomorphism $h:M\rightarrow M'$ of $\nabla$-modules over $R$ the $R'$-linear extension $h\otimes_R\textrm{id}_{R'}:M\otimes_RR'
\rightarrow M'\otimes_RR'$ is a morphism of $\nabla$-modules over $R'$. These objects form a $K$-linear Tannakian category, with respect to horizontal maps as morphisms, and with the obvious notion of directs sums, tensor products, quotients and duals. Note that we may define similar notions for the integral rings $\Gamma_+,\Gamma^{\dagger}$ and $\Gamma$ by substituting $K$-linearity with $\mathcal O$-linearity.
\end{defn}
\begin{defn} A horizontal section of a $\nabla$-module
$(M,\nabla)$ over $R$ is an $s\in M$ such that $\nabla(s)=0$. We denote the set of the latter by $M^{\nabla}$. Note that for every $s\in M^{\nabla}$ there is a unique morphism from the trivial $\nabla$-module to $(M,\nabla)$ such that the image of $1$ is $s$. Of course a $\nabla$-module over $R$ is {\it trivial} if it is isomorphic to the $n$-fold direct sum of the trivial $\nabla$-module for some $n$ (over $R$).
\end{defn}
Note that any reasonable version of Lemma \ref{solvability} is false; in fact there is a $\nabla$-module over $\mathcal E_+$ whose pull-back to $\mathcal R$ is not trivial. (In fact the basic counterexample is very simple; it corresponds to the 
differential equation $y'=y$. For a further explanation see Example 0.4.1 of \cite{Ke} on page 7.) However the analogue of the framed version (Lemma \ref{filtered_iso}) is true, at least over
$\mathcal R_+$. We are going to formulate this claim next. 
\begin{notn} Let $\underline r=(r_1,r_2,\ldots,r_n)$ be a vector consisting of positive integers, and set $r=r_1+r_2+\cdots+r_n$, as in Definition \ref{framed1}. Let $M$ be a $\nabla$-module over $R$ equipped with a filtration:
$$0=M_0\subset M_1\subset\cdots\subset M_n=M$$
by sub $\nabla$-modules such that the rank of $M_i$ over $R$ is $r_1+\cdots+r_i$. Set $r=r_1+r_2+\cdots+r_n$, and equip the trivial $\nabla$-module $T=R^{\oplus r}$ with the filtration:
$$0=T_0\subset T_1\subset\cdots\subset T_n=T,$$
where 
$$T_i=\underbrace{R\oplus R\oplus\cdots\oplus R}
_{r_1+\cdots+r_i}\oplus\underbrace{0\oplus\cdots\oplus0}
_{r_{i+1}+\cdots+r_n}.$$
Also assume that for every index $i=1,2,\ldots,n$ an isomorphism of $\nabla$-modules:
$$\phi_i:M_i/M_{i-1}\longrightarrow R^{\oplus r_i}$$
is given where $R$ is equipped with the trivial connection. We will call such objects (consisting of $(M,\nabla)$, the filtration
$M_0\subset M_1\subset\cdots\subset M_n$, and the isomorphisms $\phi_i$) {\it filtered $\nabla$-modules of signature $\underline r$}. There is a natural notion of isomorphism of filtered $\nabla$-modules of signature $\underline r$, namely, it is an isomorphism of the underlying $\nabla$-modules which maps the filtrations to each other, and identifies the isomorphisms $\phi_i$. 
\end{notn}
Now let $(M,\nabla,M_i,\phi_i)$ be a filtered $\nabla$-module of signature $\underline r$ and let $(T,T_i)$ be as above.
\begin{lemma}\label{filtered_iso2} Assume that $R=\mathcal R_+$. Then there is an isomorphism $\phi:M\to T$ of $\nabla$-modules such that $\phi(M_i)=T_i$ and the induced isomorphism
$$\phi^i:M_i/M_{i-1}\longrightarrow T_i/T_{i-1}\cong(\mathcal R_+)^{\oplus r_i}$$
is $\phi_i$ for every index $i=1,2,\ldots,n$.
\end{lemma}
It will be simpler to introduce some additional definitions before we give the proof of the lemma above.
\begin{defn}\label{framed1b} Let $\underline r=(r_1,r_2,\ldots,r_n)$ be a vector consisting of positive integers, and set $r=r_1+r_2+\cdots+r_n$. A framed $\nabla$-module of signature $\underline r$ (over $R$) is a $\nabla$-module $(M,\nabla)$ over $R$ equipped with an $R$-basis $e_1,e_2,\ldots,e_r$ of $M$ such that
$$M_i=\textrm{ the $R$-span of
$e_1,e_2,\ldots,e_{r_1+\cdots+r_i}$}$$
is a sub $\nabla$-module, and the image of $e_{r_1+\cdots+r_{i-1}+1},\ldots,e_{r_1+\cdots+r_i}$ in the quotient $M_i/M_{i-1}$ is a $k$-basis of $(M_i/M_{i-1})^{\nabla}$. There is a natural notion of isomorphism of framed $\nabla$-modules of signature
$\underline r$ in this setting, too. 
\end{defn}
\begin{proof}[Proof of Lemma \ref{filtered_iso2}] We are going to prove the claim by induction on $n$. The case $n=1$ is obvious. Assume now that the claim holds for $n-1$. Note that $(M_i/M_{i-1})^{\nabla}$ spans $M_i/M_{i-1}$ as an $\mathcal R_+$-module, since the latter is a trivial $\nabla$-module. Also note that $M$ is a free $\mathcal R_+$-module. Therefore we may choose a $\mathcal R_+$-basis $e_1,e_2,\ldots,e_r$ of $M$ such that $M_i$ is the $\mathcal R_+$-span of $e_1,\ldots,e_{r_1+\cdots+r_i}$, and $(M,\nabla)$ equipped with this basis is a framed $\nabla$-module of signature $\underline r$. By the induction hypothesis we may assume that $e_1,\ldots,e_{r_1+\cdots+r_{n-1}}$ are horizontal. Let
$\mathbf e_1,\mathbf e_2,\ldots,\mathbf e_r$ is the 1st, 2nd, etc.~basis vector of $T$. We may also assume without loss of generality that $\phi_i$ maps the image of $e_{r_1+\cdots+r_{i-1}+1},\ldots,e_{r_1+\cdots+r_i}$ under the quotient map to  the image of $\mathbf e_{r_1+\cdots+r_{i-1}+1},\ldots,\mathbf e_{r_1+\cdots+r_i}$ under the quotient map for every $i=1,\ldots,n$.

Let $C$ be the matrix of the connection $\nabla$ in the $\mathcal R_+$-basis $e_1,\ldots,e_r$, that is, for every $s_1,s_2,\ldots,s_r\in\mathcal R_+$ we have:
$$\nabla(s_1e_1+\cdots+s_re_r)=
e_1\otimes ds_1+\cdots+e_r\otimes ds_1+
(s_1e_1,\cdots,s_re_r)\cdot C,$$
where the $\cdot$ in the last term denotes the row-column multiplication with respect to the tensor product. Then $C$ is an $r\times r$ matrix with coefficients in $\Omega^1_{\mathcal R_+}$ composed of blocks $C_{ij}$ such that for every pair $(i,j)$ of indices $C_{ij}$ is an $r_i\times r_j$ matrix with coefficients in $\Omega^1_{\mathcal R_+}$, and $C_{ij}$ is the zero matrix unless $i=1$ and $j=n$.

By Lemma \ref{solvability2} there is a matrix $U$ of rank
$\underline r$ with coefficients in $\mathcal R_+$ such that
$dU=C$ and $U_{ij}$ is the zero matrix unless $i=1$ and $j=n$. Consider $\mathcal R_+$-linear map $\phi:M\to T$ given by:
$$\phi(\lambda_1e_1+\cdots+\lambda_re_r)=
(\lambda_1\mathbf e_1,\cdots,\lambda_r\mathbf e_r)\cdot(I+U)$$
for every $\lambda_1,\ldots,\lambda_r\in\mathcal R_+$, where $I$ is the $r\times r$ identity matrix and $\cdot$ denotes the row-column multiplication here. It is the isomorphism of $\nabla$-modules we are looking for.
\end{proof}
\begin{defn}\label{invariant} Now let $(M,\nabla,M_1,\ldots,M_r,\phi_1,\ldots,\phi_r)$ be a filtered $\nabla$-module of signature $\underline r$ over $\mathcal E_+$. We may choose an $\mathcal E_+$-basis $e_1,e_2,\ldots,e_r$ of $M$ such that $M_i$ is the $\mathcal E_+$-span of $e_1,\ldots,e_{r_1+\cdots+r_i}$, and $(M,\nabla)$ equipped with this basis is a framed $\nabla$-module of signature $\underline r$. By Lemma \ref{filtered_iso2} above there is an isomorphism $\phi:M
\otimes_{\mathcal E_+}\mathcal R_+\to T$ of $\nabla$-modules over $\mathcal R_+$ such that $\phi(M_i\otimes_{\mathcal E_+}\mathcal R_+)=T_i$ and the induced isomorphism
$$\phi^i:M_i\otimes_{\mathcal E_+}\mathcal R_+/
M_{i-1}\otimes_{\mathcal E_+}\mathcal R_+\cong
(M_i/M_{i-1})\otimes_{\mathcal E_+}\mathcal R_+
\longrightarrow T_i/T_{i-1}\cong\mathcal R_+^{\oplus r_i}$$
is $\phi_i\otimes_{\mathcal E_+}\!\!\textrm{id}_{\mathcal R_+}$ for every index $i=1,2,\ldots,n$. The matrix of $\phi$ in the basis $e_1\otimes_{\mathcal E_+}1,e_2\otimes_{\mathcal E_+}1,
\ldots,e_r\otimes_{\mathcal E_+}1$ is an element of
$U_{\underline r}(\mathcal R_+)$, unique up to multiplication on the right by a matrix in $U_{\underline r}(K)$, corresponding to an automorphism of the $\nabla$-module $T$ respecting its filtration and the horizontal bases on the Jordan--H\"older components, and up to multiplication on the left by a matrix in $U_{\underline r}(\mathcal E_+)$, corresponding to a change of the basis $e_1,\ldots,e_r$. We get a well-defined map from the isomorphism classes of framed  $\nabla$-modules of signature $\underline r$ over $\mathcal E_+$ into the set
$U_{\underline r}(\mathcal E_+)\backslash
U_{\underline r}(\mathcal R_+)/U_{\underline r}(K)$ of double cosets. 
\end{defn}
\begin{defn} Write $\mathcal O_n=\mathcal O/(p^{n+1})$. For a topologically finitely generated $\Gamma_+$-algebra $A$, with reductions $A_n=A/(p^{n+1})$, we let
$$\Omega^1_{A/\mathcal O}\stackrel{\textrm{def}}{=}
\varprojlim_{n\to\infty}\Omega^1_{A_n/\mathcal O_n}$$
be the module of $p$-adically continuous differentials. The limit of the differentials of $A_n$ over $\mathcal O_n$ furnishes a $p$-adically continuous differential $d:A\to\Omega^1_{A/\mathcal O}$. When $A=\Gamma_+=\mathcal O[[u]]$ then 
$\Omega^1_{\mathcal O[[u]]/\mathcal O}$ is the free $\mathcal O[[u]]$-module of rank one generated by the symbol $du$.
Let $X$ be a formally smooth $u$-adic formal scheme of finite type over Spf$(\Gamma_+)$. Then we may define the $p$-adically continuous K\"ahler differentials $\Omega^1_{X/\mathcal O}$ by patching, and it is a finite, locally free formal $\mathcal O_X$-module, equipped with a differential $d:\mathcal O_X\to\Omega^1_{X/\mathcal O}$.
\end{defn}
\begin{defn} Let $X$ be as above. A $\nabla$-module over $X$ is a pair $(M,\nabla)$, where $M$ is a finite, locally free formal
$\mathcal O_X$-module, and $\nabla$ is a connection on $M$, i.e.~an $\mathcal O$-linear map of sheaves:
$$\nabla:M\to M\otimes_{\mathcal O_X}
\Omega^1_{X/\mathcal O}$$
satisfying the Leibniz rule
$$\nabla(c\mathbf v)=c\nabla(\mathbf v)+\mathbf v\otimes dc$$
for every open $U\subset X$ and $c\in\Gamma(U,\mathcal O_X),\mathbf v\in\Gamma(U,M)$. 
\end{defn}
\begin{defn} The trivial $\nabla$-module over $X$ is just $\mathcal O_X$ equipped with the differential $d:\mathcal O_X\to\Omega^1_{X/\mathcal O}\cong\mathcal O_X\otimes_{\mathcal O_X}\Omega^1_{X/\mathcal O}$. Moreover horizontal maps of $\nabla$-modules over $X$ is defined the same way as above. We get a $K$-linear category with the usual notion of direct sums, duals and tensor products. Again we will denote by $M$ the ordered pair $(M,\nabla)$ whenever this is convenient. Finally let $M^{\nabla}$ denote the sheaf of horizontal sections of $M$:
$$\Gamma(U,M^{\nabla})\stackrel{\textrm{def}}{=}
\{s\in\Gamma(U,M)|\nabla(s)=0\}.$$
Note that $M$ is a trivial $\nabla$-module of rank $n$, that is, isomorphic to the $n$-fold direct sum of $(\mathcal O_X,d)$, if and only if $M^{\nabla}$ is the constant sheaf in rank $n$ free $\mathcal O$-modules. It is possible to define the notion of filtered and framed $\nabla$-modules in this more general context, too. We will leave the details to the reader.
\end{defn}
\begin{defn} The notion of $\nabla$-modules and framed $\nabla$-modules are natural in $X$. Let $f:X\to Y$ be a morphism of formally smooth formal schemes of finite type over Spf$(\Gamma_+)$. The morphism $f$ induces an $\mathcal O_X$-linear map $df:f^*(\Omega^1_{Y/\mathcal O})\to\Omega^1_{X/\mathcal O}$. The pull-back $f^*(M,\nabla)$ of a $\nabla$-module $(M,\nabla)$ with respect to $f$ is $f^*(M)$ equipped with the composition:
$$\xymatrix{
f^*(\nabla):f^*(M)\ar[r] & f^*(M\otimes_{\mathcal O_Y}\Omega^1_{Y/\mathcal O})\cong f^*(M)\otimes_{\mathcal O_X}f^*(\Omega^1_{Y/\mathcal O})
\ar[r]
& \Omega^1_{X/\mathcal O},}$$
where the first arrow is the pull-back of $\nabla$ with respect to $f$, and the second is $\textrm{id}_{f^*(M)}\otimes_{\mathcal O_X}df$. The pull-back of a filtered $\nabla$-module $(M,\nabla,M_1,\ldots,M_r,\phi_1,\ldots,\phi_r)$ of signature $\underline r$ on $Y$ with respect to $f$ is the pull-back $f^*(M,\nabla)$ equipped with the filtration $f^*(M_1),\ldots,f^*(M_r),f^*(\phi_1),\ldots,f^*(\phi_r)$. Since pull-back commutes with quotients and the pull-back of horizontal sections are horizontal, this construction is a filtered $\nabla$-module of signature $\underline r$ on $X$.
\end{defn}
\begin{defn} For every $X$ as above let $X(\Gamma_+)$ denote the set of sections $f:\textrm{Spf}(\Gamma_+)\to X$. Let $\mathbf M=(M,\nabla,M_1,\ldots,M_r,\phi_1,\cdots,\phi_r)$ be a filtered $\nabla$-module of signature $\underline r$ on $X$. Then for every
$f\in X(\Gamma_+)$ the pull-back of $\mathbf M$ with respect to $f$ is a filtered $\nabla$-module of signature $\underline r$ over
$\Gamma_+$. By applying the functor $\cdot\otimes_{\Gamma_+}
\mathcal E_+$ we get a filtered $\nabla$-module of signature $\underline r$ over $\mathcal E_+$. By taking isomorphism classes and using the construction in Definition \ref{invariant} we get a function
$$\int_{\mathbf M}:X(\Gamma_+)\longrightarrow
U_{\underline r}(\mathcal E_+)\backslash U_{\underline r}(\mathcal R_+)/U_{\underline r}(K)$$
which we will call the line integral of $\mathbf M$. 
\end{defn}
\begin{example}\label{nabla_global2} Let $X$ be Spf$(\mathcal O[[u,x]])$. In order to give a $\nabla$-module on $X$, it is sufficient to give a $\mathcal O$-linear map:
$$\nabla:\mathcal O[[u,x]]^{\oplus2}\longrightarrow 
\mathcal O[[u,x]]^{\oplus2}\otimes_{\mathcal O[[u,x]]}\Omega^1_{\mathcal O[[u,x]]/\mathcal O}$$
satisfying the Leibniz rule, where
$$\Omega^1_{\mathcal O[[u,x]]/\mathcal O}=
\mathcal O[[u,x]]\cdot du\oplus \mathcal O[[u,x]]\cdot dx,$$
with differential $d:\mathcal O[[u,x]]\to\Omega^1_{\mathcal O[[u,x]]/\mathcal O}$ given by:
$$d\big(\sum_{ij}a_{ij}u^ix^j\big)=\sum_{ij}
(ia_{ij}u^{i-1}x^jdu+ja_{ij}u^ix^{j-1}dx).$$
Let $e_1,e_2$ be the 1st, respectively 2nd basis vector of $\mathcal O[[u,x]]^{\oplus2}$, and let $\nabla$ be the unique connection of $\mathcal O[[u,x]]^{\oplus2}$ such that
$$\nabla(e_1)=0,\quad\nabla(e_2)=e_1\otimes\frac{dx}{1+x},$$
where $(1+x)^{-1}=\sum_{i=0}^{\infty}(-1)^ix^i$. Equipped with the frame $e_1,e_2$ this $\nabla$-module is framed of signature $(1,1)$. Let $\mathbf M$ denote this object. Note that sections of $X\to\textrm{Spf}(\mathcal O[[u]])$ are exactly continuous $\mathcal O[[u]]$-algebra homomorphisms $\psi:\mathcal O[[u,x]]\to\mathcal O[[u]]$. Every such $\psi$ is determined by $\psi(1+x)$ which must be an invertible element of $\mathcal O[[u]]$. Conversely for every $v\in\mathcal O[[u]]^*$ there is a unique such $\psi_v:\mathcal O[[u,x]]\to \mathcal O[[u]]$ with the property $\psi_v(1+x)=v$. The  pull-back of $\mathbf M$ with respect to $\psi_v$ is the framed $\nabla$-module, where $M=\mathcal O[[u]]^{\oplus 2}$, the frame $e_1,e_2$ is the 1st, respectively 2nd basis vector of $M$, and $\nabla$ is the unique connection of $M$ such that
$$\nabla(e_1)=0,\quad\nabla(e_2)=e_1\otimes\frac{dv}{v}.$$
Let $\phi:M\otimes_{\Gamma_+}\mathcal R_+\cong
\mathcal R_+^{\oplus 2}\to T\cong\mathcal R_+^{\oplus 2}$ be an isomorphism of the type considered in Definition \ref{invariant} above. Then the matrix $V$ of $\phi$ in the basis $e_1\otimes_{\Gamma_+}\mathcal R_+,e_2\otimes_{\Gamma_+}\mathcal R_+$ is
$$V=\begin{pmatrix} 1& w \\ 0 & 1\end{pmatrix}\in U_{(1,1)}(\mathcal R_+)
\textrm{ such that}$$
$$d\circ V=\begin{pmatrix} 0& dw \\ 0 & 0\end{pmatrix}=
\begin{pmatrix} 0& \frac{dv}{v} \\ 0 & 0\end{pmatrix},$$
and hence
$$dw=\frac{dv}{v}.$$
So the invariant of the framed $\nabla$-module $(M,\nabla,e_1,e_2)$ is $\log^{\dagger}(v)$, i.e.~we get that the $p$-adic line integral:
$$\int_{\mathbf M}:X(\mathcal O[[u]])\cong
\mathcal O[[u]]^*\longrightarrow 
U_{(1,1)}(\mathcal E_+)\backslash
U_{(1,1)}(\mathcal R_+)/U_{(1,1)}(K)\cong
\mathcal E_+\backslash\mathcal R_+$$
is just the $p$-adic logarithm.
\end{example}
\begin{conclrem} What we have described is just the beginning of a theory, barely setting up the formalism to state less trivial results. However the simple, but key idea is already present: we should think of line integrals as fibre functors (or isomorphisms between them), but the functor should take values in a non-trivial Tannakian category, such as $\nabla$-modules over $\mathcal E_+$. One of the main reasons to carry this theory further is to study rational points on varieties over $k((t))$ which can be seen as follows.

Let $\overline X$ denote the special fibre of $X$, that is, its base change to Spec$(k[[t]])$. It is a smooth scheme of finite type over Spec$(k[[t]])$. We have a reduction map $r:X(\Gamma_+)\to
\overline X(k[[t]])$. Assume that $(M,\nabla)$ is {\it integrable}, i.e.~the curvature of $\nabla$, defined completely analogously to the classical construction is trivial. Then the map $\int_{\mathbf M}$ factors through $r:X(\Gamma_+)\to\overline X(k[[t]])$, that is, there is a map
$$\overline X(k[[t]])\longrightarrow
U_{\underline r}(\mathcal E_+)\backslash U_{\underline r}(\mathcal R_+)/U_{\underline r}(K),$$
necessarily unique, whose composition with the reduction map $r$ is the line integral of $\mathbf M$. Clearly we need to show the following: let $s_1,s_2\in X(\Gamma_+)$ be two sections such that $r(s_1)=r(s_2)$. Then the base changes of the filtered
$\nabla$-modules $s_1^*(\mathbf M)$ and $s_2^*(\mathbf M)$ to $\mathcal E_+$ are isomorphic. The latter can be proved in the usual way, using Grothendieck's equivalence between integrable $\nabla$-modules and crystals.

The natural next step is to study $k[[t]]$-valued points of smooth projective curves over Spec$(k[[t]])$ via these line integrals. These have smooth, proper formal lifts to $\Gamma_+$, and we may look at the universal $n$-unipotent (and integrable) $\nabla$-modules on these lifts, similarly to Besser's work (see \cite{Be}). The natural expectation is that the map which we get this way is independent of the formal lift to $\Gamma_+$, it is injective on residue disks, and it is possible to prove a suitable analogue of the main result of Kim's article \cite{Kim} (Theorem 1 on page 93). Combined with the global methods of the paper \cite{La}, we are set to give a new proof of the Mordell conjecture over global function fields along the lines of Kim's method. We plan to carry out this program in a forthcoming publication. Finally, let me also add that such a theory should exists also for analytic varieties, in the sense of Huber, over the adic spectrum of $(\mathcal E_+,\Gamma_+)$, and it is perhaps the natural setting, too.
\end{conclrem}
\begin{ackn} I wish to thank Amnon Besser and Chris Lazda for some useful discussions related to the contents of this article, and the referee for his comments. The author was partially supported by the EPSRC grant P36794.
\end{ackn}


\begin{thebibliography}{99}

\bibitem[1]{Be} A.~Besser, {\it Coleman integration using the Tannakian formalism}, Math. Ann. \textbf{322} (2002), 19--48.

\bibitem[2]{Ke} K.~Kedlaya, {\it $p$-adic differential equations}, Cambridge studies in advanced mathematics \textbf{125}, Cambridge University Press, Cambridge, (2010). 

\bibitem[3]{Kim} M.-H.~Kim, {\it The unipotent Albanese map and Selmer varieties for curves}, Publ. Res. Inst. Math. Sci. \textbf{45} (2009), no. 1, 89--133. 

\bibitem[4]{La} C.~Lazda, {\it Relative fundamental groups and rational points}, Rend. Sem. Mat. Univ. Padova \textbf{134} (2015) 1--45. 
\end{thebibliography}
\end{document}